\documentclass{article}
\usepackage{amsmath}
\usepackage{amsthm}
\usepackage{amsfonts}
\usepackage{amssymb}
\usepackage{amscd}
\usepackage[all]{xy}
\usepackage{graphicx}
\usepackage{xcolor}
\usepackage{enumerate}

\usepackage{comment}
\usepackage{hyperref}

\newtheorem{theorem}{Theorem}

\newtheorem{corollary}[theorem]{Corollary}

\newtheorem{lemma}[theorem]{Lemma}

\newtheorem{proposition}[theorem]{Proposition}

\newcommand{\steq}{\succ_{stab}}

\newcommand{\Cbb}{\mathbb C}
\newcommand{\Ocal}{\mathcal O}

\newcommand{\ee}{\quad\text{ and }\quad}

\usepackage[
    backend=biber,
    style=numeric,
	sorting=nyt,
    natbib=true,
    url=false, 
    doi=true,
    backref=true
    eprint=false
]{biblatex}
\begin{document}

\title{Characterizing $(d,h)$-elliptic stable irreducible curves}
\author{Juliana Coelho\\{\scriptsize julianacoelhochaves@id.uff.br}
\and Renata Costa\\{\scriptsize renatac@id.uff.br}
}

\maketitle

\begin{abstract}
We use admissible covers to characterize irreducible stable curves that are
 $(d,h)$-elliptic, that is, that 
are limits of smooth curves admiting finite maps of degree-$d$ to smooth curves of genus $h\geq 1$. 
\end{abstract}

%\tableofcontents

\section{Introduction}

%In this paper we always work over the field of complex numbers $\mathbb{C}$.

A smooth curve $C$ is said to be $d$-gonal if it admits a degree-$d$ map to $\mathbb P^1$. 
The gonality, that is, the minimum $d$ such that $C$ is $d$-gonal, is an important numerical invariant in the study of algebraic curves. 
More generally, we say that $C$ is {$(d, h)$-elliptic} if it admits a degree-$d$ morphism to a smooth curve of genus $h$.
It follows from de Franchi's theorem that, if $d\geq 2$, then a general curve $C$ of genus $g\geq 2$ is not $(d,h)$-elliptic for any $h\geq 1$ (see \cite[Corollary XXI.8.32]{3autores}). 
It is thus an interesting problem to try to understand how such a smooth curve can degenerate, that is, what are the possible stable limits of $(d,h)$-elliptic smooth curves for $h\geq 1$.

In this paper, we deal with this issue using 
the theory of admissible covers, introduced in  \cite{HarrMum}, and  greatly generalized by many authors such as in \cite{abracortivistoli} and \cite{mochizuki}. 
This approach has already been considered 
for instance in \cite{faberpaganini}, where the authors determine the class of the so-called bielliptic locus for genus-3 stable curves, that is, the locus of $(2,1)$-elliptic curves in $\overline M_3$. 
The case $h=1$ 
is related to the classical study of elliptic integrals
and was also studied in \cite{ddelliptic2018}, albeit without the use of admissible covers, where the authors considered stable curves of genus 2 admiting two morphisms to a smooth curve of genus $1$.

%
%
%
%When considering smooth curves in families, one needs to consider singular curves as well, since smooth curves can degenerate to singular ones. 
%There are a few possible ways to extend this notion  to singular curves. 
%One way, related to the classical study of elliptic integrals, is to simply take the same definition as in the smooth case. This was done for instance 
%in \cite{kani97}
%and \cite{ddelliptic2018}. 
%The drawback here is that this definition does not work very well in families, since not all curves that are limits of smooth $(d,h)$-elliptic ones have maps of degree $d$ to some curve of genus $h$. 
%
%In this paper, we deal with this issue using 
%the theory of admissible covers, introduced in  \cite{HarrMum}, and  greatly generalized by many authors such as in \cite{abracortivistoli} and \cite{mochizuki}. 
%This approach has already been considered 
%for instance in \cite{faberpaganini}, where the authors study the so-called bielliptic locus for genus-3 stable curves. 

\subsection{Main Results}

In this paper we work over $\mathbb{C}$.
As we mentioned, we consider a stable curve $C$ to be $(d,h)$-elliptic if it is a limit of smooth ones or, equivalently, if there exists a $d$-sheeted admissible cover from a curve stably equivalent to $C$ to a curve of genus $h$. 
Rougly speaking, an admissible cover is a finite map of nodal curves satisfying some conditions, one of which is that the target is a pointed stable curve, when considered with the smooth branch points of the map. 

%On Section \ref{sec:back} we review some notation and terminology.

On Section \ref{sec:admissible} we define pseudo-admissible covers by switching this condition to the condition that every node of the target curve lies in the intersection of two components. 
On Proposition \ref{prop:construction} we present a gluing for pseudo-admissible covers in the style of those introduced in \cite{paper1}, and 
on Proposition \ref{prop:equivpseudo} we show that it is possible to characterize $(d,h)$-ellipticity using pseudo-admissible covers.

On Section \ref{sec:irred} we focus on irreducible curves. 
In Theorem \ref{thm:main} we characterize irreducible stable curves  that are $(d,h)$-elliptic for $h\geq 1$, in terms of its normalization. 
This result is similar to those of \cite{paper2}, where the case  $h=0$ was addressed for stable curves with two components.
As a consequence, we examine in Corollaries \ref{cor:cor1} and \ref{cor:cor2} the case of irreducible stable curves with only one node.

\section{Technical background}\label{sec:back}

In this paper we always work over the field of complex numbers $\mathbb{C}$.

%%%%%%%%%%%%%%%%%%%%%%%%%%%%%%%%%%%%%%%%%%%%%%%%%%%%%%%%%%%%%%%%%%%%%%%%%%%%
%\begin{comment}

A \textit{curve} $C$ is a connected, projective and reduced scheme of dimension $1$ over $\Cbb$. 
We denote by $C^{\text{sing}}$ the singular locus of $C$, and by $C^{\text{sm}}$ the smooth locus of $C$.
The \emph{genus} of $C$ is the arithmetic genus $g(C):=h^1(C,\Ocal_C)$.
A \textit{subcurve} $Y $ of $C$ is a reduced subscheme of pure dimension $1$, or equivalently, a reduced union of irreducible components of $C$. 
If $Y\subseteq C$ is a subcurve, we say $Y^{c}:=\overline{C\smallsetminus Y}$ is the \emph{complement} of $Y$ in $C$.

A \textit{nodal curve} $C$ is a curve with at most ordinary double points, called \emph{nodes}. 
A node is said to be \emph{separating} if there is a subcurve $Y$ of $C$ such that $Y\cap Y^c=\{n\}$. 
%The subcurves $Y$ and $Y^c$ are then said to be \emph{tails} of $C$ associated to the separating node $n$. A \emph{rational tail} is a tail of genus 0.
Let $C$ be a nodal curve and $\nu:{C}_\nu \rightarrow C$ be its normalization. Let 
$n$ be a node of $C$. The points $n_1, n_2 \in {C}_\nu$ such that $\nu(n_1)=\nu(n_2)=n$ are said  to be \textit{branches} of $n$.

A \emph{rational chain}  is a nodal curve of genus 0. 
We remark that a curve  is a rational chain if and only if its nodes are all separating and its components are all rational.
A \textit{cycle} is a  nodal curve $E$ whose dual graph is a cycle, that is, up to reordering, the components $E_1, \ldots, E_s$ of $E$ are such that $E_i$ meets $E_{i+1}$ at a single point, for $i=1,\ldots, s-1$, and $E_s$ meets $E_1$ at a single point. We  denote the cycle $E$ as the ordered $s$-uple $(E_1,\ldots,E_s)$. 
A \textit{rational cycle} is a cycle where all components are rational. Note that a rational cycle is a curve of genus 1.
%\begin{figure}[h!]
%    \centering
%    \includegraphics[width=0.35\linewidth]{Modelo TD PGMAT-UFF/rationalcycle.jpg}
%    \caption{Rational cycle with $s$ components.}
%    \label{fig:rationalcycle}
%\end{figure}

Let  $g$ and $n$ be non negative integers such that $2g-2+n>0$. A \textit{$n$-pointed stable curve of genus $g$} is a curve $C$ of genus $g$ together with $n$ distinct \emph{marked points} $p_1,\ldots,p_n\in C$  such that for every  smooth rational component $E$ of $C$,  
 the number of points in the intersection   $E\cap E^{c}$ plus the number of indices $i$ such that  $p_i$ lies on $E$ is at least three. 
A \textit{stable curve} is a 0-pointed stable curve.

Let $\pi: C \rightarrow B$ be a finite map between curves and let $p \in C$ and $q \in B$ be  smooth points such that $\pi(p)=q$. Let $x$ be a local parameter of $C$ around $p$ and $t$ be a local parameter of $B$ around $q$. Then, locally around $p$, the map $\pi$ is given by $t=x^e$, where $e$ is the \textit{ramification index} of $\pi$ at $p$, denoted by $e_{\pi}(p)$. Note that $1\leq e_{\pi}(p)\leq d$, where $d$ is the degree of $\pi$. 
We say that $\pi$ is \textit{ramified} at $p$, or that $p$ is a \textit{ramification point} of $\pi$, if $e_{\pi}(p)\geq 2.$ 
Moreover, we say $\pi$ is \textit{totally ramified} at $p$ if $e_{\pi}(p)=d$.
Finally, we say that $q$ is a \emph{branch point} of $\pi$ if $e_\pi(p')\geq 2$ for some $p'\in\pi^{-1}(q)$.

\section{Admissible Covers}\label{sec:admissible}

A smooth curve $C$ 
%is \emph{$d$-gonal} if it admits  admits a degree-$d$ morphism to $\mathbb{P}^1$. More generally, a smooth curve $C$ 
is \emph{$(d, h)$-elliptic} if it admits a degree-$d$ morphism to a smooth curve of genus $h$. 
In this work we define a stable curve $C$ to be \emph{$(d, h)$-elliptic} if it is a limit of smooth $(d, h)$-elliptic curves in the moduli space $\overline{\mathcal M}_g$ of stable curves. More precisely, this occurs if there is a smoothing $f:\mathcal C \rightarrow S$ whose generic fiber is a smooth $(d, h)$-elliptic curve and the special fiber is isomorphic to $C$. 
%We say a stable curve is bielliptic if it is $(2,1)$-elliptic. 

Alternativelly, $(d,h)$-ellipticity can be characterized in terms of admissible covers.
A \emph{$d$-sheeted admissible cover} consists of a finite morphism $\pi: C \rightarrow B$ of degree $d$, such that $B$ and $C$ are nodal curves and:
\begin{enumerate}
\item [$1.$] $\pi^{-1}(B^{\text{sing}})= C^{\text{sing}};$

\item[$2.$] $\pi$ is simply branched away  from $C^{\text{sing}}$, that is, over each smooth point of $B$ there exists at most one point of $C$ where $\pi$ is ramified and this point has ramification index 2;

\item[$3.$] $B$ is a stable pointed curve, when considered with the smooth branch points of $\pi$; 
%on the branch locus of $\pi$;

\item [$4.$] for every node $q$ of $B$ and every node $n$ of $C$ lying over it, the two branches of $C$ over $n$ map to the branches of $B$ over $q$ with the same ramification index.
\end{enumerate}

Let $C$ and $C'$ be nodal curves. We say that $C'$ is \textit{stably equivalent} to $C$
if $C$ can be obtained from $C'$ by contracting to a point some of the smooth rational components of $C'$ meeting the other components of $C'$ in only one or two points. 
We remark that  stable equivalence is not an equivalence relation, but rather a partial order 
and thus we write $C'\steq C$ to mean that $C'$ is stably equivalent to $C$.
Note moreover that, in this case, %if $C'$ is stably equivalent to $C$ then 
$g(C)=g(C')$ and there exists a \emph{contraction map} $\tau:C' \rightarrow C$. 
We say that a point $p' \in C'$ \textit{lies over} a point $p \in C$ if $\tau(p')=p$.

\begin{theorem}\label{thm:mochi}
A stable curve $C$ is $(d,h)$-elliptic if and only if there exists a $d$-sheeted admissible cover ${C}'\rightarrow B$, where ${C}'\steq C$ %is stably equivalent to $C$ 
and $g(B)=h$.
\end{theorem}
    \begin{proof}
Follows from  \cite{mochizuki}.        
    \end{proof}

Now we introduce a slightly relaxed but, as we'll see in Proposition \ref{prop:equivpseudo}, equivalent notion of admissibility. 
Let $C$ and $B$ be nodal curves. A finite morphism $\pi: C \rightarrow B$ of degree $d$ is called a \emph{$d$-sheeted pseudo-admissible cover} if it satisfies conditions (1), (2) and (4) of an admissible cover and also satisfies:
\begin{enumerate}
	\item [${3}'$.] the curve $B$  has no internal nodes.
\end{enumerate}

The next result is a pseudo-admissible version of \cite[Lemma 3.2]{paper1}, and we include it here for the sake of completeness. 

\begin{lemma}\label{lemma:components}%\label{l1}
    Let $C$ be a nodal curve and $\pi:{C}'\rightarrow B$ be a  $d$-sheeted pseudo-admissible cover, where $C'\steq C$. Then for each irreducible component $C_r$ of $C$ there is a unique irreducible component ${C}'_r$ of ${C}'$ such that $\tau({C}'_r)= C_r$ and $\tau|_{{C}'_r}$ is a normalization map, where $\tau\:C'\rightarrow C$ is the contraction map.
\end{lemma}
\begin{proof}
    As the components of $B$ are smooth and $\pi^{-1}(B^{sing})= C^{sing}$, then the components of ${C}'$ are also smooth. Fixing a component $C_r$ of $C$, by definition, there is an unique component ${C}'_r$ of $\tau^{-1}(C_r)$ dominating $C_r$. Since ${C}'_r$ is smooth, $\tau|_{{C}'_r}$ is a normalization map.
\end{proof}

The following two results have appeared in \cite{paper1}, albeit in the context of admissible covers to curves of genus 0.

%The following result shows that condition $(2)$ on the definition of pseudo-admissible covers can, in some sense, always be achieved given the other three conditions.

\begin{lemma}\label{lem:condition2}%\label{ref4}
    Let $\pi: C \rightarrow B$ be a degree-$d$ finite morphism of nodal curves. If $\pi$ satisfies conditions $(1), ({3}')$ and $(4)$ of $d$-sheeted pseudo-admissible covers, then there is a $d$-sheeted pseudo-admissible cover ${\pi}':{C}'\rightarrow {B}' $ such that ${C}'$ (respectively ${B}'$) is stably equivalent to $C$ (respectively $B$), contains $C$ (respectively $B$) as a subcurve and ${\pi}'|_{C}=\pi$.
\end{lemma}
\begin{proof}
The proof is {\it mutatis mutandis} that of \cite[Theorem 3.3 (a)]{paper1}. 
%
%    The proof follows the same steps as that of \cite[Theorem 3.3(a)]{paper1}. Let $q$ be a smooth point of $B$ such that
%    $$\pi^{-1}(q)=\lambda_1 m_1+ \cdots + \lambda_l m_l$$
%    where $m_1,\cdots, m_l \in C$ and $\lambda_j\geq 3$ for some $j=1,\ldots, l$ or $\lambda_j= \lambda_{{j}'}=2$ for some $j,{j}'=1,\ldots, l$ with $j\neq {j}'$. We glue:
%    \begin{enumerate}
%        \item [$\bullet$] a copy of $\mathbb{P}^1$ at $q$ denoted by $B_q$;
%
%        \item[$\bullet$] a copy of $\mathbb{P}^1$ at every $m_j$, denoted by $L_j$, mapping to $B_q$ via a degree-$\lambda_j$ map totally ramified at $m_j$, simply ramified away from $m_j$ and unramified over the branch points of $L_{{j}'}\rightarrow B_q$ for $j,{j}'=1, \ldots, l$ with $j\neq {j}'$.
%    \end{enumerate}
%
%    At the end of this process we obtain from $C$ a nodal curve ${C}'$ stably equivalent to $C$, from $B$ a nodal curve ${B}'$ stably equivalent to $B$, and a degree-$d$ map ${\pi}':{C}'\rightarrow {B}'$ given by $\pi$ when restricted to $C$, and by the maps described above when restricted to the added rational components of ${C}'$. By construction, ${\pi}$ satisfies conditions $(1), (2), (3')$ and $(4)$ of an pseudo-admissible cover.
\end{proof}

\begin{proposition}\label{prop:teo3.4b}
Let $C$ be a nodal curve and $n$ be a node of $C$ such that the 
normalization $C_n$ of $C$ at $n$ is connected.
Let $\pi\:C_n'\rightarrow B$ be a finite map of degree $d$ between nodal curves, satisfying conditions 
$(1)$, $(3')$ and $(4)$ of pseudo-admissible curves, where $C_n'\steq C_n$ .

Let $n^{(1)}, n^{(2)} \in C_n$ be the branches of $n$ and  let ${n}'^{(1)}$, ${n}'^{(2)}$ be smooth points of ${C}'_n$ lying over $n^{(1)}$ and $ n^{(2)}$, respectively.
If $\pi(n'^{(1)})=\pi(n'^{(2)})$ then there is a $d$-sheeted quasi-admissible cover $\pi'\:C'\rightarrow B'$ where $C'\steq C$, $C_n'\subset C'$ and $\pi'|_{C_n}=\pi$.
\end{proposition}
\begin{proof}
The proof is {\it mutatis mutandis} that of \cite[Theorem 3.4 (b)]{paper1}. 
\end{proof}

The next result describes a gluing of a pseudo-admissible cover, similar to those on \cite{paper1}, but more suited to study $(d,h)$-ellipticity for $h\geq 1$.

\begin{proposition}\label{prop:construction}%\label{ref5}
Let $C$ be a stable curve and 
let $C_n$ be  the normalization of $C$ at nodes 
$n_1,\ldots,n_{\delta}$. 
Assume that $C_n$ is connected and let $\pi: {C}'_n\rightarrow B$ be a finite map of degree $d$ satisfying  conditions $(1), ({3}')$ and $(4)$ of pseudo-admissible cover, where $B$ is a curve 
 of genus $h$ and ${C}'_n\steq C_n$. % is stably equivalent to $C_n$. 
 Let $n_i^{(1)}, n_i^{(2)} \in C_n$ be the branches of $n_i$ and  ${n}_i'^{(1)}$, ${n}_i'^{(2)}$ be smooth points of ${C}'_n$ lying over $n_i^{(1)}$ and $ n_i^{(2)}$, respectively, for $i=1,\ldots, \delta$.
Assume that 
\begin{enumerate}[(a)]
\item $e_\pi({n}_i'^{(1)})=e_\pi({n}_i'^{(2)})$, for $i=1,\ldots, \delta$; 
\item there exist $q_1,q_2\in B$ such that $q_1\neq q_2$ and
$\pi^{-1}(q_j)=\{n'^{(j)}_1, \ldots,n'^{(j)}_{\delta} \}$, for $j=1,2$. 
\end{enumerate} 
Then there exists a $d$-sheeted pseudo-admissible cover ${\pi}':{C}' \rightarrow {B}'$ 
where 
 $C'\steq C$, $C_n'\subset C'$ and $\pi'|_{C_n'}=\pi$.
Moreover, $g({B}')=h+1$
%such that ${C}'$ is stably equivalent to $C$, contains $C_n$ as a subcurve, and $g({B}')=h+1$, that is,
and $C$ is $(d,h+1)$-elliptic. %Furthermore, if $h=0$, then $C$ is $(d+\delta,0)$-elliptic.
\end{proposition}
\begin{proof}
Set $e_i:=e_\pi({n}_i'^{(1)})=e_\pi({n}_i'^{(2)})$,
for $i=1, \ldots, \delta$.
%Then, for $j=1,2$ we have
%$$\pi^{-1}(q_j)=e_1 n'^{(j)}_1+ e_2 n'^{(j)}_2+ \cdots + e_{\delta} n'^{(j)}_{\delta}$$
%$$\pi^{-1}(q_2)=e_1 n'^{(2)}_1+ \lambda_2 n'^{(2)}_2+ \cdots + \lambda_{\delta} n'^{(2)}_{\delta}.$$
To obtain the cover ${\pi}'$ we proceed as follows. We glue (see Figure \ref{fig:teogluing}):
    \begin{figure}[!h]
        \centering
        \includegraphics[width=4cm]{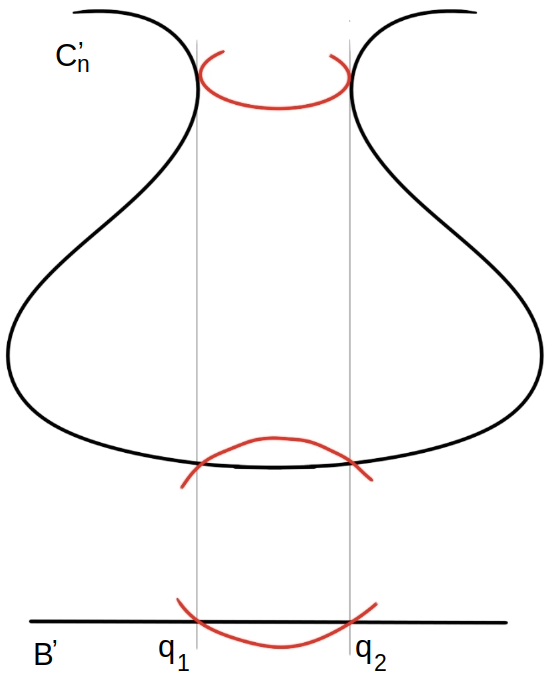}%{Modelo TD PGMAT-UFF/coberturaPseudoGrau4.png}
        \caption{$(d,h+1)$-ellipticity for $d=3$}
        \label{fig:teogluing}
    \end{figure}
    \begin{enumerate}
        \item [$\bullet$] a copy of $\mathbb{P}^1$ in $B$ passing through $q_1$ and $q_2$. Call this copy  $L$;

        \item[$\bullet$]  a copy of $\mathbb{P}^1$ in ${C}_n'$ passing through $n'^{(1)}_i$ and $n'^{(2)}_i$ mapping to $L$ via a map of degree $e_i$, totally ramified at $n'^{(1)}_i$ and $n'^{(2)}_i$, for $i=1,\ldots, \delta$. %Call this copy  $L_i$.

         \end{enumerate}

   Thus, we obtain from ${C}'_n$ a nodal curve $\widetilde {C}$ stably equivalent to ${C}$, from $B$ a nodal curve $\widetilde {B}$ such that $g(\widetilde {B})=h+1$, and the map $\widetilde {\pi} : \widetilde {C} \rightarrow \widetilde {B}$ given by the map $\pi$ when restricted to ${C}'_n$ and by the maps described above  when restricted to the added rational components of $\widetilde {C}$. By construction, ${C}'_n$ is a subcurve of $\widetilde {C}$ and the map $\widetilde {\pi}$ satisfies conditions $(1), ({3}')$ and $(4)$ of the definition of pseudo-admissible cover. 
It follows from Lemma $\ref{lem:condition2}$ that there exists a $d$-sheeted pseudo-admissible cover ${\pi}':{C}'\rightarrow {B}'$, where ${C}'$ is stably equivalent to $\widetilde {C}$ and ${B}'$  is stably equivalent to $\widetilde {B}$.
\end{proof}

We finish this section showing that the $(d,h)$-ellipticity of a stable curve can be characterized in terms of pseudo-admissible covers.

\begin{proposition} \label{prop:equivpseudo}%\label{1}
    A stable curve $C$ is $(d,h)$-elliptic if and only if there exists a $d$-sheeted pseudo-admissible cover $\pi: {C}'\rightarrow B$ where $C'\steq C$ and $g(B)= h$.
\end{proposition}
\begin{proof}
By Theorem \ref{thm:mochi},
it is enough to show that for every pseudo-admissible cover $\pi: C \rightarrow B$ there exists an admissible cover $\widetilde{\pi}:\widetilde{C} \rightarrow \widetilde{B}$
and, conversely, for every admissible cover $\widetilde{\pi}:\widetilde{C} \rightarrow \widetilde{B}$
there exists a pseudo-admissible cover $\pi: C \rightarrow B$, 
such that $B\steq \widetilde B$ and 
$C\steq \widetilde C$  the following diagram  commutes
\begin{equation}\label{eq:diagramaadmisspseudo}
\xymatrix{	
C \ar[r]^{\tau_C} \ar[d]_{\pi} 
& \widetilde C  \ar[d]^{\widetilde\pi} 
\\ B \ar[r]^{\tau_B} 
& \widetilde B }
\end{equation}
where $\tau_C$ and $\tau_B$ are the contraction maps.

First we let $\pi: C \rightarrow B$ be a $d$-sheeted pseudo-admissible cover that is not admissible. Then there exist rational components $W$ of $B$ containing at most two points that are nodes of $B$ or are smooth branch points of $\pi$. 
%less than three points that are nodes of $B$ or are smooth branch points of $\pi$.
Note that, if $L$ is a component of $C$ such that $\pi(L)=W$, then 
$L$ meets its complement in at most two points and has genus 0, by Riemann-Hurwitz.
Since the map $\pi$ satisfies conditions $(1), (2)$ and $(4)$ of the definition of admissible cover, to satisfy condition (3) it suffices to contract those components $W$ and the components of $C$ mapping to $W$. 
Thus, we obtain curves  $\widetilde{B}$ and $\widetilde{C}$ and an
induced map $\widetilde{\pi}:\widetilde{C} \rightarrow \widetilde{B}$ as in \eqref{eq:diagramaadmisspseudo}.
%Note that 
%$B\steq \widetilde B$, 
%$C\steq \widetilde C$ 
%and $\widetilde \pi$ is a $d$-sheeted admissible cover, such that the following diagram commutes
%$$\xymatrix{	
%C \ar[r]^{\tau_C} \ar[d]_{\pi} 
%& \widetilde C  \ar[d]^{\widetilde\pi} 
%\\ B \ar[r]^{\tau_B} 
%& \widetilde B }
%$$

Conversely, let $\widetilde {\pi}:\widetilde {C} \rightarrow \widetilde {B}$ be a $d$-sheeted admissible cover that is not pseudo-admissible, that is, such that $\widetilde B$ has internal nodes. 
Now note that for $n\in \widetilde C$ and $q\in\widetilde B$ such that $\widetilde \pi(n)=q$, we have that $q$ is an internal node of $B$ if and only if $n$ is an internal node of $C$.
So we let $C'$ and $B'$  be the normalization of $C$ and $B$ at its internal nodes, respectivelly. 
Then $\widetilde \pi$ induces a $d$-sheeted admissible cover $\pi':C'\rightarrow B'$ that is also quasi-admissible. 
Moreover, if $q$ is an internal node of $B$ and $\widetilde \pi^{-1}(q)=\{n_1,\ldots,n_\delta\}$ , then 
$\pi'$ satisfies the conditions of Proposition \ref{prop:construction} with respect to the branches of $n_1,\ldots,n_\delta$.
Hence, applying
Proposition \ref{prop:construction} sucessively for each internal node of $B$, 
we obtain a pseudo-admissible cover $\pi\:C\rightarrow B$ as in \eqref{eq:diagramaadmisspseudo}.
\end{proof}

\section{$(d,h)$-ellipticity for irreducible curves}\label{sec:irred}

In this section, our goal is to characterize irreducible nodal $(d,h)$-elliptic curves in terms of their normalizations.
First, we need a lemma.

\begin{lemma}\label{lem:3lemas}
Let $C$ be a nodal curve and let ${C}'$ be a nodal curve stably equivalent to $C$ with contraction map $\tau: {C}'\rightarrow C$. 
\begin{enumerate}[(a)]
\item If $Z_1$ and $Z_2$ are two connected subcurves of ${C}'$ such that $\tau(Z_1)= \tau(Z_2)$ is a subcurve of $C$, then $\tau(Z_1)= \tau(Z_2)= \tau(Z_1\cap Z_2)$.
\item If $C$ is irreducible then no two subcurves of ${C}'$ of positive genus have finite intersection.
\item Let $\pi: {C}' \rightarrow B$ be a $d$-sheeted  pseudo-admissible cover.
Assume that $C$ is irreducible and let $C_n$ be the normalization of $C$  and let $B_0 = \pi(C_n)$. Then every subcurve $Z$ of $B$ of positive genus must contain $B_0$.

In particular, if there exists a component $W$ of $B$ such that $g(W)\geq 1$, then $W=B_0$.
\end{enumerate}
\end{lemma}
\begin{proof}
First we prove (a). Set $Z:=\tau(Z_1)=\tau(Z_2)$ and assume $Z$ is a subcurve of $C$.
Recall from Lemma \ref{lemma:components} that, if $W$ is a component of $C$, 
then there exists an unique component $W'$ of $C'$ such that $\tau(W')=W$. Hence if $W\subset Z$, then we must have $W'\subset Z_1$ and also $W'\subset Z_2$.
Thus $W'\subset Z_1\cap Z_2$ showing that $Z=\tau(Z_1\cap Z_2)$. 

To show (b), let $Z_1$ and $Z_2$ be two subcurves of  ${C}'$ of positive genus
so that $\tau(Z_1)$ and $\tau(Z_2)$ are subcurves of $C$. 
Since $C$ is irreducible, we must have 
$\tau(Z_1)= \tau(Z_2)$ which, by (a), implies that $Z_1\cap Z_2$ cannot be finite.

Lastly we show (c).
Assume there exists a subcurve $Z$ of $B$ of positive genus not containing $B_0$. 
Then $C_n$ is not contained in $\pi^{-1}(Z)$. 
Let $W$ be a connected component of $\pi^{-1}(Z)$. 
Since $\pi(W)=Z$, we must have 
$g(W)\geq g(Z)\geq 1$. 
In particular, this implies that $W$ is not contracted by $\tau$. 
But since $C$ is irreducible, this means $\tau(W)=C=\tau(C_n)$ and, by (a), $W\cap C_n$ cannot be finite.  As $C_n$ is irreducible, this implies that $C_n\subset W$, a contradiction.
\end{proof}

\begin{proposition} \label{prop:degreeofrestriction}% \label{p3}
Let $C$ be an irreducible nodal curve and $C_n$ be its normalization. 
Let  $\pi:{C}'\rightarrow B$ be a $d$-sheeted pseudo-admissible cover, where ${C}'\steq C$. % is stably equivalent to $C$. 
If $g(B)\geq 1$, then $\pi|_{C_n}$ has degree $d$.
\end{proposition}
\begin{proof}
Set $B_0: = \pi(C_n)$ and 
let $e$ be the degree of the restriction  $\pi|_{C_n}$.  %$\pi_n := \pi|_{C_n} : C_n \rightarrow B_0$.
Assume $1\leq e \leq d-1$. 
Since $\pi$ is a map of degree $d$, then there exists a connected subcurve $Z$ of ${C}'$ not containing $C_n$ such that $C_n \cup Z\subset \pi^{-1}(B_0) $. 
If $e'$ is the degree of $\pi|_Z$, then $e'\geq 1$ and $e+e'\leq d$. 
We have two cases to consider.

If $g(B_0)\geq 1$ then $C_n$ and $Z$ must have positive genus as well, since they map to $B_0$. 
By Lemma \ref{lem:3lemas}, we then have that $C_n \cap Z$  is not finite
and,  since $C_n$ is irreducible, we must have $C_n\subset Z$ contradicting the choice of $Z$.

%
%\begin{figure}[!h]
%\centering
%\includegraphics[width=0.5\linewidth]{Modelo TD PGMAT-UFF/Mapa_.png}
%\caption{Case 2}
%\label{fig:case2}
%\end{figure}

Now assume $g(B_0)=0$. Then, by Lemma \ref{lem:3lemas},
every irreducible component of $B$ has genus 0  and, since $g(B)\geq 1$, there exists a rational cycle $E=(E_1,\ldots,E_s)$ in $B$ which  must contain $B_0$, say $E_1=B_0$. 
%Let $Z$ be the connected component of $\pi^{-1}(B\setminus B_0)$ containing the rational cycle $E$. 
Let
$q_1,p_1, \ldots, p_{s-1},q_2\in B$ be such that
$$E_1 \cap B_0 = {q_1},
\quad E_s \cap B_0 = {q_2},
\quad \text{ and } \quad 
E_{i}\cap E_{i+1}= {p_{i}},
\quad \text{for } i =1, \ldots, s-1. $$

Let $q^1_1 \in C_n$ and $q_1^2 \in Z$ such that
$$\pi(q_1^1)= \pi(q_1^2)= q_1$$
Since $q_1$ is a node and $\pi$ is pseudo-admissible, then $q_1^1$ and $q_1^2$ are also nodes. Therefore, there exist rational components $E_1^1$ and $E_1^2$ of $C'$  such that
$$q_1^1 = C_n \cap E_1^1\ee q_1^2= Z \cap E_1^2.$$

Note that we must have $\pi(E_1^1)= \pi(E_1^2)= E_1$. 
Furthermore, since $E_1$ has another node, namely $p_1$, then $E_1^1$ and $E_1^2$ also have nodes, say $p_1^1$ and $p_1^2$, respectively, such that
$$\pi(p_1^1)= \pi(p_1^2)= p_1. $$
Consequently, there exist rational components $E_2^1$ and $E_2^2$  of $C'$ such that
$$p_1^1 = E_1^1 \cap E_2^1\ee p_1^2= E_1^2 \cap E_2^2$$
and again
we have, $\pi(E_2^1)=\pi(E_2^2)=E_2$.

Prceeding in this manner, for $j=1,2$ we obtain rational components 
%$E_1^1,E_2^1,\ldots, E_{s-1}^1$ and
$E_1^j,E_2^j,\ldots, E_{s}^j$  of $C'$
such that  $\pi(E_i^j)=E_i$
and
$E^j_i$ meet $E^j_{i+1}$ at $p^j_{i+1}$ for $i=1,\ldots,s-1$.
Lastly, since $E_{s}$ meets $B_0$ at the node $q_2$, then 
$ E_{s}^1$ meets $C_n$ and 
$ E_{s}^2$ meets $Z$.
Hence 
$$C_n \cup E_1^1 \cup \ldots \cup E_s^1 
\ee
Z\cup E_1^2\cup \ldots \cup E_s^2$$
are two subcurves of ${C}'$ of positive genus and having no common component, contradicting Lemma \ref{lem:3lemas}, thus showing that $e=d$.
 \end{proof}   

We remark that the previous result  does not hold for admissible or pseudo-admissible covers to a curve of genus 0, as attested by \cite[Theorem 3.3 and Corollary 3.5]{paper1}.

\begin{lemma}\label{lem:lemablock} %  \label{ref3}
Let $C$ be an irreducible nodal curve and $C_n$ be its normalization. 
Let $\pi:{C}'\rightarrow B$ be
a $d$-sheeted pseudo-admissible cover, 
where ${C}'\steq C$. % is stably equivalent to $C$ and $B$ is a nodal curve of genus $h$
%Let $B_0:=\pi(C_n)$.
Assume that $g(B)\geq 1$. % and {\color{red}$g(B_0) < g(B)$}.
Let $n\in C$ be a node and denote by $n^1,n^2\in C_n$ its branches.
If $\pi(n^1)\neq \pi(n^2)$,  then there exists a cycle $Z$ in $B$ such that  $\pi(n^1)$ and $ \pi(n^2)$ are nodes of $Z$.
Moreover, we have $e_\pi(n^1)=e_\pi(n^2)$.
\end{lemma}
\begin{proof}
%Set $q_1:=\pi(n^1)$ and $q_2:=\pi(n^2)$ and 
Assume there is no  such cycle of $B$. 
Then every connected subcurve of $B$ containing $q_1:=\pi(n^1)$ and $q_2:= \pi(n^2)$ must contain $B_0:=\pi(C_n)$. 
Now, since $C'$ is stably equivalent to $C$, 
there is a rational chain $Y' \subset {C}'$ not containing $C_n$ and connecting $n^1$ to $n^2$. 
Hence, since $Y'$ is connected then so is $\pi(Y')$ and, since $\pi(Y')$ contains $q_1$ and $q_2$,
 then 
$\pi(Y')$ must contain $B_0$. 
Thus there is a component $L$ of $Y'$ such that $\pi(L)=B_0$.
Since $B_0=\pi(C_n)$, then the degree of the restriction
$\pi|_{C_n}$
must be less than
$d$, contradicting Proposition  \ref{prop:degreeofrestriction}. 

Now we show the last assertion.  
Let $Z=(B_0,E_1,\ldots,E_s)\subset B$ be the stated cycle, such that $q_1\in E_1$ and $q_2\in E_s$.
Set $Y=E_1\cup\cdots\cup E_s$, that is, $Y:=Z\setminus B_0$. 
Moreover,
set $\widetilde Y\subset Y'$ such that $C_n \cup\widetilde Y$ is a cycle in $C'$ through $n^1$ and $n^2$, 
say 
$C_n \cup\widetilde Y=(C_n,\widetilde E_1,\ldots,\widetilde E_r)$, with $n^1\in \widetilde E_1$ and $n^2\in \widetilde E_r$.
Note that $\pi(\widetilde Y)\subset Y$. 

Assume, without lack of generality, that $e_\pi(n^1)<e_\pi(n^2)$.
Then there exists a point $m_1\in \widetilde Y$ distinct from $n^1$ such that $\pi(m_1)=q_1$. 
Since $q_1$ is a node, then so is $m_1$ and there is a rational component $E_1'$ of $C'$ 
distinct from $\widetilde E_1$ containing $m_1$. 
Note that 
$\pi(E_1')=B_0$ and so $q_2\in \pi(E_1')$. Hence there exists $m_2\in E_1'$ such that $\pi(m_2)=q_2$. 
Since $q_2$ is a node, so is $m_2$ and there exists a rational component $E_2'$ of $C'$ 
distinct from $E_1'$ passing through $m_2$. 
Then $E_2'$ maps to the component of $B'$ through $q_2$ distinct from $B_0$.
Applying an argument similar to that used in the proof of 
Proposition \ref{prop:degreeofrestriction}, 
we obtain a chain in $C'$ not containing $E_1'$ and connecting $m_1$ and $m_2$. 
But this implies that $m_1$ is a triple point, a contradiction.
\end{proof}

Finally, we reach the central result of the section, which characterizes the $(d, h)$-ellipticity of stable irreducible curves in terms of that of its normalization and the configuration of its nodes.

\begin{theorem}\label{thm:main}
Let $C$ be an irreducible nodal curve and let $C_n$ be its normalization. 
Let $n_1,\ldots,n_\delta$ be the nodes of $C$ and 
let $n_1^i,n_2^i\in C_n$ be the branches of $n_i$, 
for each $i=1,\ldots,\delta$. Let $d,h\in\mathbb Z$ with $d\geq 2$ and $h\geq 1$.
%
%The curve C is $(d,h)$-elliptic if and only if $C_n$ is a smooth $(d,{h'})$-elliptic curve for some ${h'}\leq h$, with a finite map  $\pi:C_n\rightarrow B_0$ of degree $d$ and $g(B_0)={h}'$, satisfying the following conditions
%\begin{enumerate}[(a)]
%\item $e_{\pi}(n_i^1)=e_{\pi}(n_i^2),$
%for each $i=\delta_1,\ldots,\delta$; 
%\item there exist  integers 
%$$0\leq  \delta_0 < \delta_1 < \ldots < \delta_{s-1} < \delta_s = \delta$$
%and distinct points 
%$$q_{\delta_1}^1,q_{\delta_1}^2,\ldots, q_{\delta_s}^1,q_{\delta_s}^2\in B_0$$
%where $s=h-h'$, such that, up to reordering, we have
%	\begin{enumerate}
%	\item[(a.1)] $\pi(n_i^1)=\pi(n_i^2)$ for $i=1,\ldots,\delta_0$;
%	\item[(a.2)]  	$\pi^{-1}(q_{\delta_k}^1)= 
%	\{ n_{\delta_{k-1}+1}^1,\ldots, n_{\delta_k}^1 \}$
%	and 
%	$\pi^{-1}(q_{\delta_k}^2)=\{ n_{\delta_{k-1}+1}^2,\ldots, n_{\delta_k}^2 \}$
%for $k=1,\ldots, s$.
%	\end{enumerate}
%\end{enumerate}
%

The curve C is $(d,h)$-elliptic if and only if $C_n$ is a smooth $(d,{h'})$-elliptic curve for some ${h'}\leq h$, with a finite map  $\pi:C_n\rightarrow B_0$ of degree $d$ and $g(B_0)={h}'$, 
such that
there exist  integers 
$$0\leq  \delta_0 < \delta_1 < \ldots < \delta_{s-1} < \delta_s = \delta$$
and distinct points 
$$q_{\delta_1}^1,q_{\delta_1}^2,\ldots, q_{\delta_s}^1,q_{\delta_s}^2\in B_0$$
where $s=h-h'$ and, up to reordering, we have 
\begin{enumerate}[(a)]
\item $\pi(n_i^1)=\pi(n_i^2)$ for $i=1,\ldots,\delta_0$;
\item  	$\pi^{-1}(q_{\delta_k}^1)= 
\{ n_{\delta_{k-1}+1}^1,\ldots, n_{\delta_k}^1 \}$
and 
$\pi^{-1}(q_{\delta_k}^2)=\{ n_{\delta_{k-1}+1}^2,\ldots, n_{\delta_k}^2 \}$
for $k=1,\ldots, s$;
\item $e_{\pi}(n_i^1)=e_{\pi}(n_i^2),$
for each $i=\delta_0+1,\ldots,\delta$.
\end{enumerate}

\end{theorem}
\begin{proof}
Assume first that $C$ is an $(d,h)$-elliptic curve and let $\pi\:C'\rightarrow B$ be a pseudo-admissible cover of degree $d$, where $C'\steq C$ and $g(B)=h$. 
Then, by Proposition \ref{prop:degreeofrestriction},  the map $\pi|_{C_n}$ has degree $d$.

Let $\delta_0$ be the number of nodes $n_i$ such that $n_i^1$ and $n_i^2$ have the same image under $\pi$. 
If $\delta_0=\delta$, there is nothing to prove. 
Now assume $\delta_0<\delta$. 
Then there exists $i\in \{1, \ldots, \delta\}$ such that 
$\pi(n_i^1)\neq \pi(n_i^2)$. 
Set $q^1:= \pi(n_i^1) $ and $q^2:= \pi(n_i^2) $.

If, for some $i' \in \{1, \ldots, \delta\}$ we have $\pi(n_{{i}'}^1)= q^1$, then we must have $\pi(n_{{i}'}^2)= q^2$. 
Indeed, if $\pi(n_{{i}'}^2) \neq q^2$, then there exists $q \in B_0$, with $q\neq q^2$  such that $\pi(n_{{i}'}^2)= q$. 
Since $C'$ is stably equivalent to $C$  
there is a rational chain in $C'$ not containing $C_n$ connecting $n_i^1$ and $n_i^2$ and,
by Lemma \ref{lem:lemablock}, 
this chain maps
to a rational chain in $B$ not containing $B_0$ passing through $q^1$ and $q^2$.
Likewise, 
there is a rational chain in $C'$ not containing $C_n$ connecting $n_{i'}^1$ and $n_{i'}^2$ mapping 
to a rational chain in $B$ not containing $B_0$ passing through $q^1$ and $q$.
Since $q^1$ is also a point in $B_0$, then in this case $q^1$ would be a triple point of $B$ contradicting the fact that $B$ is nodal.

Let $t_1$ be the number of those $i\in\{1,\ldots,\delta\}$ such that $\pi(n_i^1)=q^1$ 
and set $\delta_1:=\delta_0+t_1$. 
Setting $q_{\delta_1}^1:=q^1$ and  $q_{\delta_1}^2:=q^2$, 
we showed that, reordering the nodes if necessary, 
$$\pi^{-1}(q^1_{\delta_1})= \{ n_{{i}}^1 \ | \ i=\delta_0+1, \ldots, \delta_1\}$$
and
$$\pi^{-1}(q^2_{\delta_1})= \{ n_{i}^2 \ | \  i=\delta_0+1, \ldots, \delta_1\}.$$

If $\delta_1=\delta$, we are done. If not, then there exists $i\in\{1,\ldots,\delta\}$ such that
$\pi(n_i^1)\neq \pi(n_i^2)$ and $\{\pi(n_i^1),\,\pi(n_i^2)\}\neq \{q_{\delta_1}^1,\,q_{\delta_1}^2\}$. In this case, proceeding as above, we obtain distinct points $q_{\delta_2}^1,q_{\delta_2}^2\in B_0$ whose pre-images under $\pi$ are of the form  above.
Proceeding in this way we obtain
$$0\leq \delta_0< \delta_1< \ldots < \delta_s = \delta$$
and distinct points $q_{\delta_1}^1,q_{\delta_1}^2,\ldots,q_{\delta_s}^1,q_{\delta_s}^2,\in B_0$
as stated, for some $s\in\mathbb N$. 
By Lemma \ref{lem:lemablock}, each pair $q_{\delta_j}^1$ and $q_{\delta_j}^2$ is associated to a cycle $Z_j$ of $B$ containing $B_0$ and having $q_{\delta_j}^1$ and $q_{\delta_j}^2$  as nodes. 
Hence
$g(B)\geq g(B_0)+s$.

Moreover, there are no other cycles in $B$. Indeed, by Lemma \ref{lem:3lemas}, any cycle $Z$ of $B$ must contain $B_0$. If $q^1,q^2\in B_0$ are the points in the intersection 
$B_0\cap (\overline{Z\setminus B_0})$, then the points in 
$(\pi|_{C_n})^{-1}(q^1)$ and $(\pi|_{C_n})^{-1}(q^1)$ 
are branches of nodes of $C$ which would imply that $q^1$ and $q^2$ would be $q^1_{\delta_j}$
and $q^2_{\delta_j}$ for some $j$. 
Therefore, we get the equality $g(B)= g(B_0)+s$. 
Finally, by Lemma \ref{lem:lemablock}, we have 
$e_{\pi}(n_i^1)=e_{\pi}(n_i^2)$
for every $i=\delta_0+1,\ldots,\delta$.

%\hrulefill 

Conversely, assume there exists a map $\pi$ as stated. Then applying Proposition \ref{prop:construction} successively $s$ times followed by Proposition \ref{prop:teo3.4b},
%Theorem 4.(b) of \cite{paper1},
 the result follows.
\end{proof}

As an example, one can classify the $(d,h)$-elliptic irreducible curves having only one node, for $h\geq 1$.

\begin{corollary}\label{cor:cor1} 
Let $C$ be an irreducible nodal curve with only one node $n$.
Let $C_n$ be the normalization of $C$ and let $n_1,n_2\in C_n$ be the branches of the nodes.
Let $d,h\in\mathbb Z$ with $d\geq 2$ and $h\geq 1$.
Then $C$ is $(d,h)$-elliptic if and only if $C_n$ satisfies one of the following conditions
\begin{enumerate}[(a)]
\item $C_n$ is $(d,h)$-elliptic with 
a finite map  $\pi:C_n\rightarrow B_0$ of degree $d$ and $g(B_0)={h}$, and $\pi(n_1)=\pi(n_2)$; or\
\item $C_n$ is $(d,h-1)$-elliptic with 
a finite map  $\pi:C_n\rightarrow B_0$ of degree $d$ and $g(B_0)={h-1}$, and $\pi(n_1)\neq\pi(n_2)$ and $\pi$ is totally ramified at  $n_1$ and $n_2$.
\end{enumerate}
\end{corollary}
\begin{proof}
Follows from  Theorem \ref{thm:main}.
\end{proof}

\begin{corollary}\label{cor:cor2} 
Let $C$ be an irreducible nodal curve with only one node $n$ and let $C_n$ be it normalization.
If $C_n$ is hyperelliptic and has genus at least 2, then $C$ is either hyperelliptic or bielliptic, but not both.
\end{corollary}
\begin{proof}
Follows Theorem \ref{thm:main}, since the hyperelliptic map  for $C_n$ is unique.
\end{proof}

\printbibliography

\end{document}